\documentclass[10pt]{amsart}

\usepackage{amsmath}
\usepackage{amssymb}
\usepackage{graphicx}

\newtheorem{theorem}{Theorem}[section]
\newtheorem{corollary}[theorem]{Corollary}
\newtheorem{lemma}[theorem]{Lemma}
\newtheorem{proposition}[theorem]{Proposition}
\theoremstyle{definition}
\newtheorem{definition}[theorem]{Definition}
\newtheorem{example}[theorem]{Example}
\newtheorem{remark}[theorem]{Remark}

\numberwithin{equation}{section}

\title[On Multifunctions Defined Implicitly By Set-Valued Inclusions]
{On Differential Properties of Multifunctions Defined Implicitly By Set-Valued Inclusions}

\author[A. Uderzo]{Amos Uderzo}

\address[A. Uderzo]{Dept. of Mathematics and Applications, University
of Milano - Bicocca, Milano, Italy}
\email{{\tt amos.uderzo@unimib.it}}



\date{\today}

\newcommand{\Pb}{\mathbb P}
\newcommand{\R}{\mathbb R}
\newcommand{\N}{\mathbb N}

\newcommand{\X}{\mathbb X}
\newcommand{\Y}{\mathbb Y}
\newcommand{\Uball}{{\mathbb B}}
\newcommand{\Usfer}{{\mathbb S}}

\newcommand{\dom}{{\rm dom}\, }
\newcommand{\graph}{{\rm graph}\,}

\newcommand{\nullv}{\mathbf{0}}

\newcommand{\cl}{{\rm cl}\, }
\newcommand{\cone}{{\rm cone}\, }
\newcommand{\bd}{{\rm bd}\, }
\newcommand{\inte}{{\rm int}\, }

\newcommand{\stardif}{\hbox{${*\over {}}$}}

\newcommand{\SVI}{{\rm SVI}\,}

\newcommand{\Solv}{{\mathcal S}}

\newcommand{\Fder}{\widehat{\rm D}}

\newcommand{\Fsubd}{\widehat{\partial}}

\newcommand{\stsl}[1]{|\nabla #1|}
\newcommand{\pastsl}[1]{|\nabla_x #1|}

\newcommand{\ball}[2]{{\rm B}(#1, #2)}

\newcommand{\lip}[2]{{\rm lip}\, #1(#2)}
\newcommand{\dist}[2]{{\rm dist}\left(#1,#2\right)}
\newcommand{\exc}[2]{{\rm exc}(#1,#2)}

\newcommand{\Tang}[2]{{\rm T}(#1;#2)}
\newcommand{\Ncone}[2]{{\rm \widehat{N}}(#1;#2)}
\newcommand{\haus}[2]{{\rm haus}(#1,#2)}

\newcommand{\excFC}[2]{\nu_{F,C}(#1,#2)}
\newcommand{\core}[2]{|#1\stardif#2|}
\newcommand{\Gder}[2]{{\rm D}#1(#2)}
\newcommand{\Coder}[2]{{\rm \widehat{D}^*}#1(#2)}
\newcommand{\ind}[2]{\iota(#1;#2)}

\newcommand{\Preder}[3]{H_{#1}(#2;#3)}

\begin{document}

\begin{abstract}
In the present paper, several properties concerning generalized derivatives
of multifunctions implicitly defined by set-valued inclusions are studied
by techniques of variational analysis. Set-valued inclusions are problems formalizing
the robust fulfilment of cone constraint systems, whose data are affected by
a ``crude knowledge" of uncertain elements, so they can not be casted in traditional generalized
equations.

The focus of this study in on the first-order behaviour of the solution mapping
associated with a parameterized set-valued inclusion,
starting with Lipschitzian properties and then considering its graphical derivative.
In particular, a condition for the Aubin continuity of the solution mapping is
established in terms of outer prederivative of the set-valued mapping defining
the inclusion. A large class of parameterized set-valued inculsions is singled out,
whose solution mapping turns out to be convex. Some relevant consequences on the graphical derivative
are explored. In the absence of that, formulae for the inner and outer approximation
of the graphical derivative are provided by means of prederivatives of the
problem data. A representation useful to calculate the coderivative of the solution
mapping is also obtained via the subdifferential of a merit function.
\end{abstract}

\maketitle




\section{Introduction and problem statement}

The concept of implicit function has been devised to enable calculations
and, more generally, to deal with solutions of parameterized problems
that can not be explicitly solved. Historically, the study of conditions under which
a smooth equation system determines its variables as a function of
parameters, as well as the continuity and differentiability properties
of the function so defined, was been the theme of fruitful speculations
in classic analysis.
While the first implicit function theorem, as modernly meant, seems to be due
to Cauchy, as a matter of fact functions defined implicitly by equations
can be traced back to earlier works authored by the founding fathers
of differential calculus (for detailed historical remarks, see
\cite[Commentary to Chapter 1]{DonRoc14} and references therein).
When specific features of modern variational analysis, with the
acceptance of set-valued mappings as basic mathematical objects,
led to address more general class of problems, such as inequality
and cone constraint systems, variational inequalities and equilibrium
problems, similar questions have been posed with reference to the
multifunction counterpart of the original concept of implicit function. As a result,
a comprehensive theory of multifunctions implicitly defined by
generalized equations came up, which has been brought to a high level
of development in the last decades or so
(see, among other, \cite{BorZhu05,DonRoc14,DurStr12,GfrOut16,Ioff17,LedZhu99,LeTaYe08,
Mord94,Mord94b,Mord06,NgaThe04,NgTrTh13,Robi76,Robi79,Robi91,RocWet98,Schi07}).
Although generalized equations are a problem format able to subsume
the vast majority of mathematical conditions encountered in optimization
and variational analysis, the treatment of constraint systems arising in
robust optimization seems to be left out by such a formalism.
After the seminal paper \cite{BenNem98}, robust optimization considers
constraint systems of the form
$$
   f(x,\omega)\in C,
$$
for given $f:\R^n\times\R^k\rightrightarrows\R^m$ and $C\subset\R^m$,
where $x\in\R^n$ represents the decision vector whereas $\omega\in\R^k$
the data element of the problem. In many decision environments, described
and discussed by concrete examples in \cite{BenNem98}, while the knowledge
of the data may be partly or fully uncertain, reducing to the crude fact
that $\omega$ belongs to a given uncertain set $\Omega\subseteq\R^k$,
on the other hand the constraint system $f(x,\omega)\in C$ must be
satisfied independently of the actual realization of $\omega\in\Omega$.
This feature of the problem leads to the concept of robust feasibility,
formalized by the set-valued inclusion
\begin{equation}
   \Phi_f(x)=f(x,\Omega)=\{f(x,\omega)\ |\ \omega\in\Omega\}
   \subseteq C
\end{equation}
and to the related notion of robust optimal solution to uncertain
optimization problems. It is worth noting that the same problem format
arises when considering vector optimization problems, which are
characterized by a criterion function affected by uncertain data elements
(see \cite{KhTaZa15}).

In spite of the clear motivation and the urgent demand for skills on
the aforemention issue, to the best of the author's knowledge the
solution analysis of set-valued inclusions is still very little explored.
In fact, an error bound estimate was achieved in \cite{Cast99}, under
a $C$-concavity assumption, by techniques of convex analysis.
A different approach to error bounds and to solution existence is proposed
in \cite{Uder19}, which is based on the $C$-increase behaviour,
a sort of set-valued counterpart of the decrease principle (see \cite{BorZhu05}).
Conditions for solution existence, global error bounds and
characterizations of the contingent cone to the solution set are also
investigated in \cite{Uder20}, following the convex analysis
approach initiated in \cite{Cast99}. Besides, a perturbation analysis
of the solution set to parameterized set-valued inclusions has been
started in \cite{Uder20b}. More precisely,
given a set-valued mapping
$F:P\times X\rightrightarrows Y$ and a nonempty closed set $C\subset Y$,
the following set-valued inclusion problem is considered there:
find $x\in X$ such that
$$
  F(p,x)\subseteq C,   \leqno (\SVI_p)
$$
The above class of set-valued inclusions implicitly defines
the solution mapping $\Solv:P\rightrightarrows X$ as
$$
  \Solv(p)=\{x\in X\ |\ F(p,x)\subseteq C\}.
$$
The paper \cite{Uder20b} contains sufficient conditions for several
quantitative forms of semicontinuity of $\Solv$, including those known
as Lipschitz lower semicontinuity and calmness in the variational
analysis literature.

The present paper carries on this research line, focussing instead
on the Aubin and Lipschitz continuity of $\Solv$, as well as on
its first-order behaviour. In particular, a first attempt of studying
the graphical derivative of $\Solv$ is undertaken. In consideration of
the fact that, as recognized in \cite[Chapter 9]{RocWet98},
``{\it the notion of Lipschitz continuity [\dots] singles out a class of
functions which, although not necessarily differentiable, have a property
akin to differentiability in furnishing estimates of magnitudes, if
not the directions, of changes}" (the same could be repeated for
multifunctions), the subject of the investigations here reported
can be regarded as an introduction to the sensitivity analysis of
problems $(\SVI_p)$.

The contents of the paper are organized in the subsequent sections
according to the following outline.
In Section \ref{Sect:2} basic notions and tools needed for implementing
the study of the subject by a variational technique are recalled.
Since a part of the analysis refers to concepts that find in merely
metric spaces their natural setting, this section is arranged in
two subsections, presenting material in the absence or in the presence
of a vector structure. The two results furnished with a full proof
in this section capture the main ideas behind the approach of study proposed in
the paper.
In Section \ref{Sect:3} the metric space formulation of results
about the behaviour of the solution mapping to parameterized set-valued
inclusion is presented.
In particular, a condition for the Aubin continuity of this
set-valued mapping is established in terms of nondegeneracy
of the strong slope of a merit function. It is clear that
error bound estimates play a fundamental role here.
Section \ref{Sect:4} contains the main findings of the paper
and some discussion of them. A condition for the Aubin continuity
of implict multifunctions defined by $(\SVI_p)$ in normed vector spaces
is expressed in terms of problem data,
by means of the outer prederivatives, and some consequence for
its graphical derivative are discussed. The class of $C$-concave
parameterized set-valued inclusions is shown to exhibit a convex
solution mapping, which is thereby protodifferentiable, with a convex process
as a graphical derivative. In the absence of $C$-concavity,
some formulae for the inner and outer approximation
of the graphical derivative are presented. Elements for a
representation of the coderivative of the solution mapping
conclude this section.


\section{Preliminaries}    \label{Sect:2}

\subsection{Variational analysis tools in metric spaces}

Throughout the present subsection $(P,d)$, $(X,d)$ and $(Y,d)$
denote metric spaces.
Given a function $\varphi:X\longrightarrow\R\cup\{\pm\infty\}$ and $\alpha\in\R\cup\{\pm\infty\}$,
define $[\varphi\le\alpha]=\varphi^{-1}([-\infty,\alpha])$, $[\varphi>\alpha]=\varphi^{-1}
((\alpha,+\infty])$ and $[\varphi=\alpha]=\varphi^{-1}(\alpha)$. The symbol
$\dom\varphi=\varphi^{-1}(\R)$ stands for the domain of $\varphi$.
Given $x\in X$ and $r\ge 0$,
the closed ball centered at $x$ with radius $r$ is denoted by $\ball{x}{r}=
[d(\cdot,x)\le r]$. If $S\subseteq X$, define $\dist{x}{S}=\inf_{z\in S}d(x,z)$
and $\ball{S}{r}=[\dist{\cdot}{S}\le r]$. The symbol $\ind{\,\cdot\,}{S}:
X\longrightarrow\{0,+\infty\}$ denotes the indicator function of the set $S$.
Given two subsets $A,\, B\subseteq X$, the excess
of $A$ over $B$ is indicated by $\exc{A}{B}=\sup_{a\in A}\dist{a}{B}$,
whereas the Pompeiu-Hausdorff distance between $A$ and $B$ by $\haus{A}{B}=
\max\{\exc{A}{B},\exc{B}{A}\}$. The topological closure, the interior and the boundary
of a set $S\subseteq X$ are denoted by $\cl S$, $\inte S$, and $\bd S$, respectively.
Given a set-valued mapping $\Phi:X\rightrightarrows Y$, $\dom\Phi=\{x\in X\
|\ \Phi(x)\ne\varnothing\}$ and $\graph\Phi=\{(x,y)\in X\times Y\ |\ y\in\Phi(x)\}$
stand for the domain and the graph of $\Phi$, respectively.
Given $C\subseteq Y$, the upper inverse image of $C$ through $\Phi$
is indicated by $\Phi^{+1}(C)=\{x\in X\ |\ \Phi(x)\subseteq C\}$.
The acronyms l.s.c. and u.s.c. stand for lower and upper semicontinuous,
respectively.

Given a set-valued mapping $\Phi:X\rightrightarrows Y$
and a closed set $C\subset Y$, the solution set of the set-valued inclusion
$$
  \Phi(x)\subseteq C,   \leqno (\SVI)
$$
namely the set $\Phi^{+1}(C)$, can be conveniently reformulated via level/sublevel
sets of the merit function $\nu_{\Phi,C}:X\longrightarrow\R\cup
\{\pm\infty\}$, defined through the excess as being
\begin{equation}     \label{eq:defnuPhi}
  \nu_{\Phi,C}(x)=\exc{\Phi(x)}{C}=\sup_{y\in\Phi(x)}\dist{y}{C}.
\end{equation}
To this aim, observe that, if $\dom\Phi=X$, then it is $[\nu_{\Phi,C}\ge 0]=X$.
Thus, in this case the following equality holds
$$
   \Phi^{+1}(C)=[\nu_{\Phi,C}=0].
$$
More in general, in the case $X\backslash\dom\Phi\ne\varnothing$, if
accepting the usual convention $\sup\varnothing=-\infty$, then one has
$X\backslash\dom\Phi=[\nu_{\Phi,C}=-\infty]$ and hence
$$
   \Phi^{+1}(C)=[\nu_{\Phi,C}\le 0].
$$
Since the elements of $X\backslash\dom\Phi$ are trivial solutions
of $(\SVI)$, the equality $\dom\Phi=X$ will be maintained as a
standing assumption in the rest of the paper.

Besides, it is useful to note that, whenever $\Phi$ takes bounded
values, one has $[\nu_{\Phi,C}<+\infty]=X$.

\begin{remark}   \label{rem:lscnu}
As one expects, the function $\nu_{\Phi,C}$ defined through $(\ref{eq:defnuPhi})$
inherits various properties from $\Phi$. For the purposes of the present
investigations, it is useful to recall that if $\Phi:P\rightrightarrows X$
is l.s.c. at $x_0\in X$, then $\nu_{\Phi,C}$ is l.s.c. at the same point.
If $\Phi$ is Hausdorff $C$-u.s.c. (in particular, u.s.c.) at $x_0$, then
$\nu_{\Phi,C}$ is u.s.c. at the same point
(see \cite[Lemma 2.3]{Uder19}).
Notice that, whenever $\Phi:P\rightrightarrows X$ is l.s.c. on $X$, then
$\Phi^{+1}(C)$ is a closed (possibly empty) subset of $X$.
\end{remark}

The concept of Lipschitz continuity can be adapted in different ways if
referred to set-valued mappings. In the context of mappings with bounded
values, it seems to be natural to extend immediately the notion valid for functions
via the Hausdorff distance. Accordingly, a set-valued mapping
$\Phi:P\rightrightarrows X$ between metric spaces is said to be Lipschitz
continuous with rate $\kappa>0$ in a subset $S\subseteq P$ if
$$
  \haus{\Phi(p_1)}{\Phi(p_2)}\le\kappa d(p_1,p_2),\quad\forall p_1,\, p_2
  \in S.
$$
In more general contexts of interest to variational analysis, a more
general notion gained a wide attention, inasmuch as it revealed to be
intertwined with profound phenomena of regularity.
This notion\footnote{Introduced under the name of ``pseudo-Lipschitz"
property in \cite{Aubi84}, later on it became popular as Lipschitz-likeness
or Aubin property/continuity.}, playing a crucial role in the present paper,
is recalled below.

\begin{definition}[Aubin continuity]
A set-valued mapping $\Phi:P\rightrightarrows X$ between metric spaces
is said to be {\it Aubin continuous} at $(\bar p,\bar x)\in\graph\Phi$
with rate $\kappa>0$ if there exist positive $\delta$ and $r$ such that
\begin{equation}     \label{in:Aubcont}
  \dist{x}{\Phi(p_1)}\le\kappa d(p_1,p_2),\quad\forall p_1,\, p_2\in
  \ball{\bar p}{\delta},\ \forall x\in \Phi(p_2)\cap\ball{\bar x}{r}.
\end{equation}
The value
\begin{equation}
   \lip{\Phi}{\bar p,\bar x}=\inf\{\kappa>0\ |\ \exists\delta,\, r>0\
   \hbox{ such that inequality } (\ref{in:Aubcont}) \hbox{ holds}\,\}
\end{equation}
is called {\it modulus of Aubin continuity} of $\Phi$ at $(\bar p,\bar x)$.
\end{definition}

Another Lipschitzian property for set-valued mappings, which is worth
being mentioned in connection with the subject of the present investigations,
is calmness. The behaviour that it postulates can be obtained from
condition $(\ref{in:Aubcont})$, by fixing $p_1=\bar p$, so it
results in a property weaker than Aubin continuity.

From inclusion $(\ref{in:Aubcont})$, by taking $p_2=\bar p$ and $p_1=p
\in\ball{\bar p}{\delta}$, one gets
$$
  \dist{x}{\Phi(p)}\le\kappa d(p,\bar p),\quad\forall p\in
  \ball{\bar p}{\delta},\ \forall x\in \Phi(\bar p)\cap\ball{\bar x}{r},
$$
which implies, in particular, the existence of $\ell>0$ such that
\begin{equation}     \label{in:Liplowsemicont}
   \Phi(p)\cap\ball{\bar x}{\ell d(p,\bar p)}\ne\varnothing,
   \quad\forall p\in\ball{\bar p}{\delta}.
\end{equation}
The behaviour of $\Phi$ obtained in $(\ref{in:Liplowsemicont})$ as
a further consequence of the Aubin continuity of $\Phi$, which can
be regarded as a local version of inner semicontinuity, is called
Lipschitz lower semicontinuity in \cite{KlaKum02}.

A standard technique for establishing solution existence (solvability)
and estimates of the distance from the solution set (error bounds)
to inequalities in metric spaces relies on a quantitative employment
of metric completeness via the Ekeland variational principle,
which enables to replace iteration schemes.
Such a technique can be fruitfully implemented by means of the notion of
strong slope of a function $\varphi:X\longrightarrow\R\cup\{\pm\infty\}$
at a point $x_0$ in a metric space $X$, defined as
\begin{eqnarray*}
  \stsl{\varphi}(x_0)=\left\{\begin{array}{ll}
  0, & \hbox{ if $x_0$ is a local minimizer of $\varphi$}, \\
  \displaystyle\limsup_{x\to x_0}{\varphi(x_0)-\varphi(x)\over d(x,x_0)}, &
  \hbox{ otherwise.}
  \end{array}
  \right.
\end{eqnarray*}
After \cite{AzCoLu02,Ioff00}, the usage of this tool has become
standard in variational analysis. It is well know that, whenever
$X$ is a normed vector space $(\X,\|\cdot\|)$ and $\varphi$ is
Fr\'echet differentiable at $x_0\in\dom\varphi$, with derivative
$\Fder\varphi(x_0)\in\X^*$, then it holds $\stsl{\varphi}(x_0)
=\|\Fder\varphi(x_0)\|$, while if $\varphi$ is convex on $\X$ and continuous
at $x_0$, then $\stsl{\varphi}(x_0)=\dist{\nullv^*}{\partial\varphi(x_0)}$,
where $\partial\varphi(x_0)$ denotes the subdifferential of $\varphi$
at $x_0$ is the sense of convex analysis.

In what follows, to deal with set-valued mappings depending on a
parameter, a partial variant of the strong slope with respect to
the variable $x$ of a function $\varphi:P\times X\longrightarrow\R
\cup\{\pm\infty\}$, defined on the product of metric spaces, at
$(p_0,x_0)\in P\times X$, will be considered, which is defined as
\begin{eqnarray*}
  \pastsl{\varphi}(p_0,x_0)=\left\{\begin{array}{ll}
  0, & \hbox{ if $(p_0,x_0)$ is a local minimizer of $\varphi$}, \\
  \displaystyle\limsup_{x\to x_0}{\varphi(p_0,x_0)-\varphi(p_0,x)\over d(x,x_0)}, &
  \hbox{ otherwise.}
  \end{array}
  \right.
\end{eqnarray*}

A sufficient condition for the behaviour of the solution set to parameterized
inequalities, able to trigger the forthcoming analysis, can be expressed in terms of
partial strong slope. The next lemma, whose role is fundamental according
to the approach here followed, extends to a metric space setting an analogous tool
of analysis valid in more structured contexts (see \cite[Theorem 3.6.3]{BorZhu05}).

\begin{lemma}[Parametric basic lemma]   \label{lem:parbaslem}
Let $P$, $X$ and $Y$ be metric spaces
and let $(\bar p,\bar x)\in P\times X$. Suppose that $X$ and a function
$\nu:P\times X\longrightarrow[0,+\infty]$ satisfy the following conditions:

(i) $(X,d)$ is metrically complete;

(ii) $\nu(\bar p,\bar x)=0$;

(iii) the function $p\mapsto \nu(p,\bar x)$ is u.s.c. at $\bar p$;

(iv) there exists $\delta_1>0$ such that, for every $p\in
\ball{\bar p}{\delta_1}$, each function $x\mapsto \nu(p,x)$ is l.s.c.
on $X$;

(v) there exists $\delta_2>0$ such that
$$
   \sigma=\inf\{\pastsl{\nu}(p,x)\ |\ (p,x)\in [\ball{\bar p}{\delta_2}
   \times \ball{\bar x}{\delta_2}]\cap [\nu>0]\}>0.
$$

\noindent Then, there exist positive $\eta$ and $\zeta$ such that

(t) $[\nu(p,\cdot)=0]\cap\ball{\bar x}{\eta}\ne\varnothing$, for every
$p\in\ball{\bar p}{\zeta}$;

(tt) the following estimate holds
\begin{equation}    \label{in:parerbo}
    \dist{x}{[\nu(p,\cdot)=0]}\le{\nu(p,x)\over \sigma},\quad\forall
    (p,x)\in\ball{\bar p}{\zeta}\times \ball{\bar x}{\eta}.
\end{equation}
\end{lemma}

\begin{proof}
(t) Take an arbitrary $\widetilde{\sigma}\in (0,\sigma)$. As it is
$\nu(\bar p,\bar x)=0$, then according to hypothesis (iii) there
exists $0<\delta_3<\min\{\delta_1,\, \delta_2\}$ such that
\begin{equation}   \label{in:sigmadelta3}
  \nu(p,\bar x)<{\widetilde{\sigma}\delta_2\over 3},\quad\forall
  p\in\ball{\bar p}{\delta_3}.
\end{equation}
Set $\zeta=\delta_3$ and fix an arbitrary $p\in\ball{\bar p}{\zeta}$. Then
consider the corresponding function $\nu(p,\cdot):X\longrightarrow [0,+\infty]$.
Since it is $\zeta<\delta_1$, by hypothesis (iv) $\nu(p,\cdot)$ is
l.s.c. on $X$ (and bounded from below).
Moreover, because of inequality $(\ref{in:sigmadelta3})$, clearly it is
$$
 \nu(p,\bar x)\le\inf_{x\in X}\nu(p,x)+{\widetilde{\sigma}\delta_2\over 3}.
$$
Thus, by the Ekeland variational principle, which can be invoked
owing to hypothesis (i), there exists $x_p\in X$ such that
$$
   \nu(p,x_p)\le\nu(p,\bar x)<{\widetilde{\sigma}\delta_2\over 3},
$$
\begin{equation}   \label{in:xpxdist}
  d(x_p,\bar x)\le {\delta_2\over 3},
\end{equation}
and
$$
  \nu(p,x_p)<\nu(p,x)+\widetilde{\sigma} d(x,x_p),\quad\forall
  x\in X\backslash\{x_p\},
$$
whence one readily obtains
$$
  \pastsl{\nu}(p,x_p)=\max\left\{\limsup_{x\to x_p}{\nu(p,x_p)-\nu(p,x)
  \over d(x,x_p)}, \, 0\right\}\le \widetilde{\sigma}<\sigma.
$$
Notice that, as it is $\zeta<\delta_2$, it is true that
$(p,x_p)\in [\ball{\bar p}{\delta_2}\times \ball{\bar x}{\delta_2}]$.
This fact entails that $\nu(p,x_p)=0$ for, if it were $\nu(p,x_p)>0$,
one would find contradicted the hypothesis (v).
So, one is forced to admit that $\nu(p,x_p)=0$.
Therefore, taking into account inequality $(\ref{in:xpxdist})$,
it suffices to set $\eta=\delta_2/3$ in order to get
\begin{equation}     \label{cap:soldisteta}
   x_p\in [\nu(p,\cdot)=0]\cap\ball{\bar x}{\eta}\ne\varnothing.
\end{equation}
By the arbitrariness of $p\in\ball{\bar p}{\zeta}$, the above argument
proves the assertion (t).

\vskip.5cm

\noindent (tt) Fix $(p,x)\in \ball{\bar p}{\zeta}\times\ball{\bar x}{\eta}
\cap [\nu>0]$, where $\zeta$ and $\eta$ are as in the proof of (t),
and set $r_{p,x}=\nu(p,x)/\widetilde{\sigma}$, with $\widetilde{\sigma}
\in (0,\sigma)$.

Let us consider first the case $r_{p,x}\ge 2\eta$. In such an event,
since as a consequence of $(\ref{cap:soldisteta})$ it holds
$$
  \dist{x}{[\nu(p,\cdot)=0]}\le d(x,\bar x)+\dist{\bar x}{[\nu(p,\cdot)=0]}
  \le 2\eta,
$$
then inequality $(\ref{in:parerbo})$ is immediately proved.

Let us consider now the case $r_{p,x}<2\eta$. Take a positive $\widetilde{r}$
in such a way that $r_{p,x}<\widetilde{r}<2\eta$.
Since it is $\nu(p,x)<\widetilde{r}\widetilde{\sigma}$ and $\nu(p,x)
\le\inf_{z\in X}\nu(p,z)+\widetilde{r}\widetilde{\sigma}$, one can
employ the same argument as for the proof of the assertion (t),
thus getting $\widetilde{x_p}\in X$ such that
\begin{equation}    \label{in:tilder}
   d(\widetilde{x_p},x)\le\widetilde{r}
\end{equation}
and
\begin{equation}     \label{in:stsltildexp}
   \pastsl{\nu}(p,\widetilde{x_p})\le\widetilde{\sigma}<\sigma.
\end{equation}
Since on account of inequality $(\ref{in:tilder})$ it holds
$$
   d(\widetilde{x_p},\bar x)\le d(\widetilde{x_p},x)+d(x,\bar x)
   \le\widetilde{r}+\eta\le 3\eta=\delta_2,
$$
the only way to avoid a contradiction following from inequality
$(\ref{in:stsltildexp})$ is to admit that $\widetilde{x_p}\in
[\nu(p,\cdot)=0]$. Consequently, it results in
$$
  \dist{x}{[\nu(p,\cdot)=0]}\le d(x,\widetilde{x_p})\le
  \widetilde{r}.
$$
As the argument leading to the last inequality works for every
$\widetilde{r}\in (r_{p,x},2\eta)$, one can deduce that
$$
  \dist{x}{[\nu(p,\cdot)=0]}\le r_{p,x}
$$
and hence, by arbitrariness of $\widetilde{\sigma}\in (0,\sigma)$,
one can achieve the inequality $(\ref{in:parerbo})$.
This completes the proof.
\end{proof}


\subsection{Variational analysis tools in normed vector spaces}

Throughout the current subsection, $(\X,\|\cdot\|)$ and $(\Y,\|\cdot\|)$
denote normed vector spaces. The null vector in a normed vector space
is indicated by $\nullv$. Define $\Uball=\ball{\nullv}{1}$ and
$\Usfer=\bd\Uball$. Given a set $S\subseteq\X$, $\cone S$ stands for
the conic hull of $S$. The (topological) dual space of $\X$ is denoted
by $\X^*$ and its null element by $\nullv^*$, while the bilinear form
defining the duality pairing between normed vector spaces is indicated
by $\langle\cdot,\cdot\rangle$. The acronym p.h. stands for positively
homogeneous.

The next remark collects some properties of the excess over a cone, which
may occur in a vector space setting, in view of a subsequent employment
through the function $\nu_{\Phi,C}$.

\begin{remark}[Excess over a cone]  \label{rem:provectexc}
Let $C\subseteq\Y$ be a closed, convex cone.

(i) For any $A,\, B\subseteq\Y$ and $t\in (0,+\infty)$,
as a straightforward consequence of the sublinearity of the function
$y\mapsto\dist{y}{C}$, one has
$\exc{A+B}{C}\le \exc{A}{C}+\exc{B}{C}$ and $\exc{tA}{C}\le t\exc{A}{C}$.

(ii) For any $A\subseteq\Y$ it holds $\exc{A+C}{C}=\exc{A}{C}$
(see \cite[Remark 2.1(iv)]{Uder19}).

(iii) Since it is $\dist{y}{C}\le\|y||$ for every $y\in\Y$, then
given any $r>0$, it holds $\exc{r\Uball}{C}\le r$.

(iv) For any $y\in\Y\backslash C$ and $r>0$, it holds
$\dist{y+r\Uball}{C}=\dist{y}{C}+r$ (see \cite[Lemma 2.1]{Uder19}).

(v) Let $S\subseteq\Y$ be such that $S\not\subseteq C$. Then,
for every $r>0$, it holds
$$
  \exc{S+r\Uball}{C}=\sup_{y\in S}\dist{y+r\Uball}{C}=
  \sup_{y\in S\backslash C}[\dist{y}{C}+r]=\exc{S}{C}+r.
$$

(vi) It is easy to see that for any $A,\, B\subseteq\Y$,
it holds $\exc{A}{C}\le \exc{A}{B}+\exc{B}{C}$.
\end{remark}

Given two nonempty subsets $K,\, S\subseteq\Y$, their $\ast$-difference
(a.k.a. Pontryagin difference) is defined as
$$
  K\stardif S=\{y\in\Y\ |\ y+S\subseteq K\}.
$$
It is readily seen that $\nullv\in K\stardif S$ iff
$S\subseteq K$. In what follows, several conditions will be expressed
in terms of the following quantity
\begin{equation}     \label{eq:defcore}
   \core{K}{S}=\sup\{r>0\ |\ r\Uball\subseteq K\stardif S\},
\end{equation}
which can be regarded as a measure of how much the set $S$ is inner to $K$
(for more details on the  $\ast$-difference, see for instance \cite{RubVla00}).

Given a set-valued mapping $\Phi:\X\rightrightarrows\Y$ between
normed vector spaces, several notions of first-order approximations
of $\Phi$ can be found in variational analysis, which reveal to be
suitable in connection with the present approach of study. Let
$x_0\in\dom\Phi$. After \cite{Ioff81}, a p.h. set-valued mapping
$\Preder{\Phi}{x_0}{\cdot}:\X\rightrightarrows\Y$ is said to be an outer prederivative
of $\Phi$ at $x_0$ if for every $\epsilon>0$  there exists $\delta>0$
such that
$$
  \Phi(x)\subseteq \Phi(x_0)+\Preder{\Phi}{x_0}{x-x_0}+\epsilon\|x-x_0\|\Uball,
  \quad\forall x\in\ball{x_0}{\delta}.
$$
In contrast with \cite{Pang11}, a p.h. set-valued mapping
$\Preder{\Phi}{x_0}{\cdot}:\X\rightrightarrows\Y$
is said to be an outer prederivative of $\Phi$ at $x_0$ if for every $\epsilon>0$
there exists $\delta>0$ such that
$$
  \Phi(x_0)+\Preder{\Phi}{x_0}{x-x_0}\subseteq\Phi(x)+\epsilon\|x-x_0\|\Uball,
  \quad\forall x\in\ball{x_0}{\delta}.
$$
For expanding the discussion about prederivatives, the reader may refer to
\cite{Ioff81,Pang11}.

Graphical differentiation represents a different way of approximating set-valued
mappings. It is based on the notion of conical approximation of sets.
Given a nonempty set $S\subseteq\Y$ and $y\in S$, let $\Tang{S}{y}$ denote,
in particular, the contingent cone to $S$ at $y$. Recall that $\Tang{S}{y}$
is always a closed cone and, whenever $S$ is convex, $\Tang{S}{y}$ too is convex
and can be represented as
\begin{equation}    \label{eq:Tangrepconvex}
  \Tang{S}{y}=\cl[\cone(S-y)]
\end{equation}
(see, for instance, \cite[Proposition 11.1.2(d)]{Schi07}). Besides, in view of
the technique of proof employed in a subsequent section, the following variational
characterization of the contingent cone to any set $S$ at $y$ will be helpful
\begin{equation}    \label{eq:Tangchar}
  \Tang{S}{y}=\left\{v\in\Y\ |\ \liminf_{t\downarrow 0}
  {\dist{y+tv}{S}\over t}=0\right\}
\end{equation}
(see \cite[Proposition 11.1.5]{Schi07}).

The graphical (contingent)
derivative of $\Phi$ at $(x_0,y_0)\in\graph\Phi$ is the set-valued mapping
$\Gder{\Phi}{x_0,y_0}:\X\rightrightarrows\Y$ defined via the graphical relation
$$
  \graph\Gder{\Phi}{x_0,y_0}=\Tang{\graph\Phi}{(x_0,y_0)}.
$$
Namely, the fact that $v\in\Gder{\Phi}{x_0,y_0}(z)$ means that there exist sequences
$(z_n)_n$ in $\X$, with $z_n\longrightarrow z$, $(v_n)_n$ in $\Y$, with $v_n\longrightarrow v$,
and $(t_n)_n$ in $(0,+\infty)$, with $t_n\downarrow 0$, as $n\to\infty$, such that
\begin{equation}    \label{in:defgrphder}
  y_0+t_nv_n\in\Phi(x_0+t_nz_n),\quad\forall n\in\N.
\end{equation}
From the very definition, one readily sees that $\Gder{\Phi}{x_0,y_0}$ is
a p.h. set-valued mapping.

If for each $v\in\Gder{\Phi}{x_0,y_0}(z)$ and choice of the sequence $(t_n)_n$ in $(0,+\infty)$,
with $t_n\downarrow 0$, there exist sequences $(z_n)_n$ in $\X$, with $z_n\longrightarrow z$,
and $(v_n)_n$ in $\Y$, with $v_n\longrightarrow v$, such that inclusion $(\ref{in:defgrphder})$
holds, then $\Phi$ is said to be protodifferentiable at $(x_0,y_0)\in\graph\Phi$.
Detailed accounts on graphical differentiation can be found in \cite{AubFra09,
DonRoc14,RocWet98,Schi07}.

An aspect which should be the subject of meditation is that, while outer/inner
approximations provided by prederivatives refer to an element $x_0\in\dom\Phi$
and consider the whole set $\Phi(x_0)$, graphical derivatives
refer to an element $(x_0,y_0)\in\graph\Phi$ and are affected only by the
local geometry of $\Phi$ near $(x_0,y_0)$.

Other convenient derivative-like objects for set-valued mappings are
coderivatives. They can be introduced via normal cones to the graph
of set-valued mappings. Accordingly, the Fr\'echet coderivative of
$\Phi:\X\rightrightarrows\Y$ at $(x_0,y_0)\in\graph\Phi$ is the
set-valued mapping $\Coder{\Phi}{x_0,y_0}:\Y^*\rightrightarrows\X^*$
defined by
$$
  \Coder{\Phi}{x_0,y_0}(y^*)=\{x^*\ |\ (x^*,-y^*)\in
  \Ncone{\graph\Phi}{(x_0,y_0)}\},
$$
where, if $S\subseteq\X\times\Y$ and $w_0\in S$, the subset
$$
   \Ncone{S}{w_0}=\left\{w^*\in\X^*\times\Y^*\ |\
   \limsup_{S\ni w\to w_0}{\langle w^*,w- w_0\rangle\over
    \|w-w_0\|}\le 0\right\}
$$
denotes the Fr\'echet normal cone to $S$ at $w_0$. For more
material on coderivative, see \cite{BorZhu05,DonRoc14,Mord06,RocWet98,Schi07}.
In view of a subsequent employment, let us recall the following equality
linking the Fr\'echet normal cone (and hence the coderivative) with
the Fr\'echet subdifferential via the indicator and the distance
function:
\begin{equation}    \label{eq:Fnconrep}
   \Ncone{S}{w_0}=\Fsubd\ind{\cdot}{S}(w_0)=
   \bigcup_{\kappa>0}\kappa\,\Fsubd\dist{\cdot}{S}(w_0),
\end{equation}
where
$$
  \Fsubd\varphi(x_0)=\left\{x^*\in\X^*\ |\
   \liminf_{x\to x_0}{\varphi(x)-\varphi(x_0)-
   \langle x^*,x-x_0\rangle\over\|x-x_0\|}\ge 0\right\}
$$
denotes the Fr\'echet subdifferential of a function $\varphi:
\X\longrightarrow\R\cup\{\pm\infty\}$ at $x_0\in\dom\varphi$
(see, for instance, \cite[Corollary 1.96]{Mord06}).

The next proposition explains how outer prederivatives of a set-valued
mapping $\Phi$ can be exploited for estimating the strong slope of the
function $\nu_{\Phi,C}$, at points which fail to be a solution of the
set-valued inclusion defined by $\Phi$ and $C$: roughly speaking, such
first-order approximations of $\Phi$ must
admit a direction, along which their values are strictly inner to $C$.

\begin{proposition}  \label{pro:stsloutder}
Let $\Phi:\X\rightrightarrows\Y$ be a set-valued mapping between
Banach spaces, let $C\subseteq\Y$ be a closed, convex cone and let
$x_0\in\X\backslash \Phi^{+1}(C)$. Suppose that

(i) $\Phi$ is l.s.c. on $x_0$;

(ii) $\Phi$ admits $\Preder{\Phi}{x_0}{\cdot}$ as an outer
prederivative at $x_0$;

(iii) it holds
\begin{equation}     \label{eq:defsigmaH}
   \sigma_H(x_0)=\sup_{u\in\Usfer} \core{C}{\Preder{\Phi}{x_0}{u}}>0.
\end{equation}
\noindent Then, the following estimate holds
\begin{equation}     \label{in:stslnusigma}
  \stsl{\nu_{\Phi,C}}(x_0)\ge\sigma_H(x_0).
\end{equation}
\end{proposition}

\begin{proof}
In the light of Remark \ref{rem:lscnu},
by virtue of hypothesis (i), the function ${\nu_{\Phi,C}}$ turns out
to be l.s.c. on $x_0$. Since it is $x_0\in\X\backslash\Phi^{+1}(C)$, one has
${\nu_{\Phi,C}}(x_0)>0$. Then, there exists $\delta>0$ such that
${\nu_{\Phi,C}}(x)>0$ for every $x\in\ball{x_0}{\delta}$.
According to hypothesis (ii), fixed any $\epsilon\in(0,\sigma_H(x_0))$
there exists $\delta_\epsilon\in (0,\delta)$ such that
\begin{equation}    \label{in:prederx0v}
  \Phi(x_0+v)\subseteq \Phi(x_0)+\Preder{\Phi}{x_0}{v}+\epsilon
  \|v\|\Uball, \quad\forall v\in\delta_\epsilon\Uball.
\end{equation}
By virtue of hypothesis (iii), there exists $u_\epsilon\in\Usfer$
such that
$$
  \core{C}{\Preder{\Phi}{x_0}{u_\epsilon}}>\sigma_H(x_0)-\epsilon,
$$
and hence, recalling definition $(\ref{eq:defcore})$, there exists $r_\epsilon
>\sigma_H(x_0)-\epsilon$ such that
$$
  \Preder{\Phi}{x_0}{u_\epsilon}+r_\epsilon\Uball\subseteq C.
$$
Since the set-valued mapping $\Preder{\Phi}{x_0}{\cdot}$ is positively
homogeneous and $C$ is a cone, the last inclusion entails
\begin{equation}    \label{in:predercoreC}
  \Preder{\Phi}{x_0}{tu_\epsilon}+tr_\epsilon\Uball\subseteq C,
  \quad\forall t>0.
\end{equation}
By combining inclusions $(\ref{in:prederx0v})$ and
$(\ref{in:predercoreC})$, one finds
\begin{eqnarray*}
  \Phi(x_0+tu_\epsilon)+tr_\epsilon\Uball &\subseteq &
  \Phi(x_0)+[\Preder{\Phi}{x_0}{tu_\epsilon}+tr_\epsilon\Uball]
  +t\epsilon\Uball  \\
  &\subseteq & \Phi(x_0)+C+t\epsilon\Uball,\quad
  \forall t\in (0,\delta_\epsilon).
\end{eqnarray*}
Consequently, it results in
\begin{equation}     \label{in:excPhixx0}
  \exc{\Phi(x_0+tu_\epsilon)+tr_\epsilon\Uball}{C}\le
  \exc{\Phi(x_0)+C+t\epsilon\Uball}{C},\quad
  \forall t\in (0,\delta_\epsilon).
\end{equation}
On the other hand, by Remark \ref{rem:provectexc}(ii) and (v),
recalling that $\Phi(x_0)\not\subseteq C$ as well as $\Phi(x_0+tu_\epsilon)
\not\subseteq C$ for every $t\in (0,\delta_\epsilon)$, because
$\delta_\epsilon<\delta$, so $x_0+tu_\epsilon\in [\nu_{\Phi,C}>0]$,
one obtains
$$
  \exc{\Phi(x_0+tu_\epsilon)+tr_\epsilon\Uball}{C}=
  \nu_{\Phi,C}(x_0+tu_\epsilon)+tr_\epsilon
$$
and
$$
  \exc{\Phi(x_0)+C+t\epsilon\Uball}{C}=\nu_{\Phi,C}(x_0)+t\epsilon.
$$
In the light of inequality $(\ref{in:excPhixx0})$, the above equalities
yield
$$
  {\nu_{\Phi,C}(x_0)-\nu_{\Phi,C}(x_0+tu_\epsilon)\over t}\ge
  r_\epsilon-\epsilon,\quad\forall t\in (0,\delta_\epsilon),
$$
whence one gets
$$
   \sup_{x\in\ball{x_0}{t}\backslash\{x_0\}}
   {\nu_{\Phi,C}(x_0)-\nu_{\Phi,C}(x)\over \|x-x_0\|}
   \ge  r_\epsilon-\epsilon>\sigma_H(x_0)-2\epsilon,
   \quad\forall t\in (0,\delta_\epsilon).
$$
Thus, one obtains
$$
  \stsl{\nu_{\Phi,C}}(x_0)=\lim_{t\downarrow 0}
  \sup_{x\in\ball{x_0}{t}\backslash\{x_0\}}
   {\nu_{\Phi,C}(x_0)-\nu_{\Phi,C}(x)\over \|x-x_0\|}\ge
   \sigma_H(x_0)-2\epsilon.
$$
By arbitrariness of $\epsilon$ the estimate in $(\ref{in:stslnusigma})$
follows from the last inequality.
\end{proof}

The proof of Proposition \ref{pro:stsloutder} should help to understand a possible
reading of the crucial condition $(\ref{eq:defsigmaH})$: it prescribes a behaviour
of $\Phi$ near $x_0$, which results in the existence of a descent direction for
$\nu_{\Phi,C}$, with a rate controlled by $\sigma_H(x_0)$.

\begin{remark}
As a caveat regarding condition $(\ref{eq:defsigmaH})$, it must be noticed
that such a requirement can be satisfied only if $\inte C\ne\varnothing$.
\end{remark}


\section{Lipschitzian behaviour in metric spaces}    \label{Sect:3}

Pursuing the research line presented in \cite{Uder20b}, the
study of properties of the solution mapping $\Solv$ to $(\SVI_p)$
will be carried out by means of the merit function $\nu_{F,C}:
P\times X\longrightarrow\R\cup\{\pm\infty\}$, given by
\begin{equation}     \label{eq:defnuFC}
  \excFC{p}{x}=\exc{F(p,x)}{C}=\sup_{y\in F(p,x)}\dist{y}{C}.
\end{equation}
Such an approach allows one to embed the analysis of the quantitative stability
properties of $\Solv$ into a framework, which is suitable for applying the
parametric basic lemma.

In what follows, consistently with the material exposed in Section \ref{Sect:2},
$\dom F=P\times X$ will be kept as a standing assumption, so it is
$[\excFC{p}{x}\ge 0]=P\times X$.

\begin{proposition}[Parametric solvability and error bound]   \label{pro:solverboS}
Given a parameterized problem $(\SVI_p)$, let $(\bar p,\bar x)\in P\times X$.
Suppose that:

(i) $(X,d)$ is metrically complete;

(ii) $\bar x\in\Solv(\bar p)$;

(iii) the set-valued mapping $p\leadsto F(p,\bar x)$ is Hausdorff
$C$-u.s.c. at $\bar p$;

(iv) there exists $\delta_1>0$ such that for every $p\in
\ball{\bar p}{\delta_1}$ each set-valued mapping $x\leadsto F(p,x)$
is l.s.c. on $X$;

(v) there exists $\delta_2>0$ such that
$$
   \sigma_\nabla=\inf\{\pastsl{\nu_{F,C}}(p,x)\ |\ (p,x)\in [\ball{\bar p}{\delta_2}
   \times \ball{\bar x}{\delta_2}]\backslash\graph\Solv\}>0.
$$

\noindent Then, there exist positive $\eta$ and $\zeta$ such that

(t) $\Solv(p)\cap\ball{\bar x}{\eta}\ne\varnothing$, for every
$p\in\ball{\bar p}{\zeta}$;

(tt) the following estimate holds
\begin{equation}    \label{in:parerboFC}
    \dist{x}{\Solv(p)}\le{\nu_{F,C}(p,x)\over \sigma_\nabla},
    \quad\forall
    (p,x)\in\ball{\bar p}{\zeta}\times \ball{\bar x}{\eta}.
\end{equation}

\end{proposition}

\begin{proof}
It suffices to apply the parametric basic lemma (Lemma \ref{lem:parbaslem})
with $\nu=\nu_{F,C}$, after having noted that, by Remark \ref{rem:lscnu}, under
the current hypotheses $\nu_{F,C}(\cdot,\bar x)$ is u.s.c. at
$\bar p$ and each function $\nu_{F,C}(p,\cdot)$ is l.s.c. on $X$,
for every $p$ near $\bar p$.
Then, it remains to remeber that $\Solv(p)=[\nu_{F,C}(p,\cdot)=0]$.
\end{proof}

It is worth remarking that, as a consequence of assertion (t), one gets that
each problem $(\SVI_p)$, for every $p$ near $\bar p$, does admit a solution.
In other words, it is $\bar p\in\inte\dom\Solv$.
The error bound inequality $(\ref{in:parerboFC})$ says that $\nu_{F,C}$
works as a residual in estimating the distance from the solution set to
$(\SVI_p)$. While to compute the term in the left side of $(\ref{in:parerboFC})$
one needs to find explicitly the solutions to $(\SVI_p)$, what might
be considerably difficult, the residual in
the right-side is expressed in terms of problem data, so is expected to
be more easily computed.

An important consequence of the above error bound can be established upon
an additional hypothesis on $F$. This leads to the next result about
the Lipschitzian behaviour of $\Solv$.

\begin{theorem}[Aubin continuity of $\Solv$]     \label{thm:AubconSol}
Given a parameterized problem $(\SVI_p)$, let $(\bar p,\bar x)\in P\times X$.
Suppose that all the hypotheses of Proposition \ref{pro:solverboS} are
in force and suppose that

(vi) there exist positive $\tau$ and $s$ such that for every $x\in\ball{\bar x}{s}$
each set-valued mapping $p\leadsto F(p,x)$ is Lipschitz with rate $\ell$ in
$\ball{\bar p}{\tau}$.

\noindent Then, $\bar p\in\inte\dom\Solv$, $\Solv$ is Aubin continuous
at $(\bar p,\bar x)$ and the following estimates holds
\begin{equation}
   \lip{\Solv}{\bar p,\bar x}\le {\ell\over\sigma_\nabla}.
\end{equation}
\end{theorem}

\begin{proof}
From assertion (t) in Proposition \ref{pro:solverboS} it follows
that $\bar p\in\inte\dom\Solv$ and that there exist positive $\zeta,\,
\eta$ such that the estimate $(\ref{in:parerboFC})$ holds true.
So, setting $\delta=\min\{\zeta,\tau\}$ and $r=\min\{\eta,s\}$,
by remembering the inequality in Remark \ref{rem:provectexc}(vi),
one obtains
\begin{eqnarray*}
  \nu_{F,C}(p_1,x)&\le&\exc{F(p_1,x)}{F(p_2,x)}+\exc{F(p_2,x)}{C} \\
  &\le& \ell d(p_1,p_2),\quad\forall p_1,\, p_2\in\ball{\bar p}{\delta},
  \ \forall x\in\ball{\bar x}{r}\cap \Solv(p_2).
\end{eqnarray*}
Thus, on account of inequality $(\ref{in:parerboFC})$, it results in
\begin{equation}     \label{in:AubconSolv}
  \dist{x}{\Solv(p_1)}\le {\ell\over\sigma_\nabla}d(p_1,p_2),\quad
  \forall p_1,\, p_2\in\ball{\bar p}{\delta},
  \ \forall x\in\ball{\bar x}{r}\cap \Solv(p_2),
\end{equation}
which shows that condition $(\ref{in:Aubcont})$ is satisfied with
rate $\kappa=\ell/\sigma_\nabla$. The estimate in the assertion
comes as a direct consequence of the definition of modulus of
Aubin continuity, in the light of inequality $(\ref{in:AubconSolv})$.
\end{proof}

A comparison of Theorem \ref{thm:AubconSol} with \cite[Theorem 3.3]{Uder20b}
should be useful in order to evaluate its impact. The latter result provides
a sufficient condition for the calmness of $\Solv$, under a milder
set of hypotheses. Nonetheless, inasmuch as Aubin continuity implies
calmness, Theorem \ref{thm:AubconSol} establishes an enhanced
Lipschitzian property of $\Solv$.

In the same vein, it is worth noting that, since Aubin continuity
implies Lipschitz lower semicontinuity as seen in Section \ref{Sect:2},
Theorem \ref{thm:AubconSol} contains a sufficient condition also
for the Lipschitz lower semicontinuity of $\Solv$. Of course, in
establishing a stronger Lipschitz behaviour, the invoked hypotheses
are stronger than the ones in \cite[Theorem 3.1]{Uder20b}, which is
a condition specifically tailored for Lipschitz lower semicontinuity.

The next example aims at illustrating the crucial role played by
the condition in hypothesis (v) of Proposition \ref{pro:solverboS}.

\begin{example}
Let $P=X=\R$ and $Y=\R^m$ be endowed with their usual (Euclidean) metric
structure. Consider the parameterized set-valued inclusion $(\SVI_p)$
with data $F:\R\times\R\rightrightarrows\R^m$ and $C$ defined by
$$
  F(p,x)=\{y=(y_1,\dots,y_m)\in\R^m\ |\ \min_{i=1,\dots,m}y_i\ge x^2-p\}
  \ \hbox{ and }\ C=\R^m_+.
$$
Fixed $\bar p=0$, it is clear that $\bar x=0\in\Solv(0)$. More generally,
since it is readily seen that $F(p,x)\subseteq\R^m_+$ iff $x^2-p\ge 0$,
for the problem under consideration the solution set-valued mapping
$\Solv:\R\rightrightarrows\R$ can be computed explicitly, resulting in
\begin{eqnarray*}
  \Solv(p)=\left\{\begin{array}{ll}
  \R, & \forall p\in (-\infty,0], \\
  \\
  (-\infty,-\sqrt{p}]\cup [\sqrt{p},+\infty), & \forall p\in (0,+\infty).
    \end{array}
  \right.
\end{eqnarray*}
The set-valued mapping $p\leadsto F(p,0)$ is evidently Hausdorff $\R^m_+$-u.s.c.
(though failing to be u.s.c.) at $\bar p=0$. Moreover, as a consequence of
the continuity of the function $x\mapsto x^2-p$, each set-valued mapping
$x\leadsto F(p,x)=(x^2-p)\{(1,\dots,1)\}+\R^m_+$ is l.s.c. on $\R$.
From the definition of $F$, one deduces
\begin{eqnarray*}
  \excFC{p}{x}=\left\{\begin{array}{ll}
  0, & \forall (p,x)\in \graph\Solv, \\
  \\
  \sqrt{m}|x^2-p|, & \forall (p,x)\in (\R\times\R)\backslash\graph\Solv.
    \end{array}
  \right.
\end{eqnarray*}
Consequently, fixed any $(p,x)\in (\R\times\R)\backslash\graph\Solv$,
one obtains
$$
  \pastsl{\nu_{F,C}}(p,x)=\left|{\partial\over\partial x}\sqrt{m}(p-x^2)
  \right|=2\sqrt{m}|x|.
$$
Thus, if considering the $\graph\Solv$ near its point $(0,0)$, one sees that for
any fixed $\delta>0$ there exists $(p,0)\in [\ball{0}{\delta}\times
\ball{0}{\delta}]\backslash\graph\Solv$ such that $\pastsl{\nu_{F,C}}(p,0)=0$.
This leads to conclude that $\sigma_\nabla=0$, so hypothesis (v) of
Proposition \ref{pro:solverboS} in this case is not satisfied.
One can check that, whereas the nonemptiness in assertion (t) actually takes
place, the function $\nu_{F,C}$ fails to work as a residual for $\dist{x}{\Solv(p)}$.
Indeed, taking $x=0$, one finds
\begin{eqnarray}    \label{eq:dist0Spex}
  \dist{0}{\Solv(p)}=\left\{\begin{array}{ll}
  0, &  \forall p\in (-\infty,0], \\
  \\
  \sqrt{p}, & \forall p\in (0,+\infty).
    \end{array}
  \right.
\end{eqnarray}
Clearly, the inequality
$$
  \dist{0}{\Solv(p)}=\sqrt{p}\le\kappa p=\kappa\excFC{p}{0},
  \quad\forall p\in (0,\zeta),
$$
can not be true for any choice of positive $\kappa$ and $\zeta$.
For a similar reason, it is worth noting that the expression
$(\ref{eq:dist0Spex})$ reveals that $\Solv$ fails to be Aubin
continuous at $(0,0)$. Nevertheless, one can check by using its definition
that the set-valued mapping $p\leadsto F(p,x)$ is Lipschitz continuous
with rate $\ell=1$ in $\R$, for every $x\in\R$.
\end{example}

\vskip1cm


\section{First-order analysis in normed vector spaces}    \label{Sect:4}

Unless otherwise stated, throughout the present section, $(\Pb,\|\cdot\|)$ and $(\Y,\|\cdot\|)$
will be normed vector spaces, whereas $(\X,\|\cdot\|)$ will be assumed to be a Banach
space. Whenever considered, the product space $\Pb\times\X$ will be
assumed to be equipped with the max-norm $\|(p,x)\|=\max\{\|p\|,\, \|x\|\}$.
Moreover, in view of the employment of condition $(\ref{eq:defsigmaH})$, it will
be assumed $\inte C\ne\varnothing$.

In a normed vector space setting,
a sufficient condition for error bounds and the Aubin continuity of a multifunction,
defined implicitly by $(\SVI_p)$, can be established in terms of outer
prederivative as follows.

\begin{proposition}      \label{pro:vecterboAub}
Given a parameterized set-valued inclusion $(\SVI_p)$, let $(\bar p,\bar x)
\in\Pb\times \X$.
Suppose that:

(i) $\bar x\in\Solv(\bar p)$;

(ii) the set-valued mapping $p\leadsto F(p,\bar x)$ is Hausdorff $C$-u.s.c.
at $\bar p$;

(iii) there exists $\delta_1>0$ such that for every $p\in
\ball{\bar p}{\delta_1}$ each set-valued mapping $x\leadsto F(p,x)$
is l.s.c. on $\X$;

(iv) there exists $\delta_2>0$ such that, for every $p\in
\ball{\bar p}{\delta_2}$, each set-valued mapping $x\leadsto F(p,x)$
admits an outer prederivative $H_{F(p,\cdot)}(x;\cdot):\X\rightrightarrows\Y$
at each point $x\in \ball{\bar p}{\delta_2}$;

(v) it holds
$$
   \sigma_H(\bar p,\bar x)=\inf\{\sigma_{H_{F(p,\cdot)}}(x)\ |\
   (p,x)\in [\ball{\bar p}{\delta_2}
   \times \ball{\bar x}{\delta_2}]\backslash\graph\Solv\}>0,
$$
where $\sigma_{H_{F(p,\cdot)}}(x)$ is defined as in $(\ref{eq:defsigmaH})$.

\noindent Then, there exist positive $\eta$ and $\zeta$ such that
the estimate $(\ref{in:parerboFC})$ holds true with $\sigma_\nabla$ replaced
by $\sigma_H(\bar p,\bar x)$.

\noindent If, in addition,

(vi) there exist positive $\tau$ and $s$ such that for every $x\in\ball{\bar x}{s}$
each set-valued mapping $p\leadsto F(p,x)$ is Lipschitz with rate $\ell$ in
$\ball{\bar p}{\tau}$,

\noindent then $\Solv$ is Aubin continuous
at $(\bar p,\bar x)$ and the following estimates holds
\begin{equation}
   \lip{\Solv}{\bar p,\bar x}\le {\ell\over\sigma_H(\bar p,\bar x)}.
\end{equation}
\end{proposition}

\begin{proof}
If $(p,x)\not\in\graph\Solv$, then $F(p,x)\not\subseteq C$, or, equivalently,
$x\in\X\backslash F^{+1}(p,\cdot)(C)$. Thus, under the
above  hypotheses it is possible to apply Proposition \ref{pro:stsloutder}.
Consequently, for every $(p,x)\in [\ball{\bar p}{\delta_2} \times
\ball{\bar x}{\delta_2}]\backslash\graph\Solv$ one obtains
$$
  \pastsl{\nu_{F,C}}(p,x)\ge\sigma_{H_{F(p,\cdot)}}(x),
$$
which implies
$$
  \sigma_\nabla\ge\sigma_H(\bar p,\bar x).
$$
The last inequality, on account of hypothesis (v), enables one to
apply Proposition \ref{pro:solverboS} and, under the hypothesis (vi),
Theorem \ref{thm:AubconSol}.
\end{proof}

Proposition \ref{pro:vecterboAub} has the following consequence on the
graphical derivative of $\Solv$, which is relevant to the sensitivity
analysis of $(\SVI_p)$.

\begin{corollary}[Lipschitz continuity of $\Gder{\Solv}{\bar p,\bar x}$]
Given a parameterized set-valued inclusion $(\SVI_p)$, let $(\bar p,\bar x)
\in\Pb\times \X$. Under the hypotheses (i) -- (vi) of Proposition \ref{pro:vecterboAub},
$\dom\Gder{\Solv}{\bar p,\bar x}=\Pb$ and $\Gder{\Solv}{\bar p,\bar x}:\Pb
\rightrightarrows\X$ is Lipschitz continuous on $\Pb$ with rate $\kappa\le
\ell/\sigma_H(\bar p,\bar x)$.
\end{corollary}

\begin{proof}
The thesis follows from the Aubin continuity of $\Solv$ in the light
of \cite[Exercise 9.49]{RocWet98}. Indeed, a perusal of the argument
suggested there certifies that the finite-dimensional setting does
not affect the reasoning.
\end{proof}

Other useful properties of $\Gder{\Solv}{\bar p,\bar x}$ can be
established in the presence of a specific geometric property of
$F$, which appeared in connection with the study of set-valued
already in \cite{Cast99}.

\begin{definition}[$C$-concavity]
A set-valued mapping $\Phi:\X\rightrightarrows\Y$ between normed
vector spaces is called {\it $C$-concave} in $\X$, where $C$ is a
convex cone in $\Y$, if it holds
$$
  \Phi(tx_1+(1-t)x_2)\subseteq t\Phi(x_1)+(1-t)\Phi(x_2)+C,
  \quad\forall x_1,\, x_2\in\X,\ \forall t\in [0,1].
$$
\end{definition}

A remarkable class of $C$-concave set-valued mappings emerging in the
context of robust convex optimization is singled out below.

\begin{example}
Let $f:\X\times\Omega\longrightarrow\Y$ be a given mapping, with $\Omega
\ne\varnothing$, and let $C\subseteq\Y$ be a convex cone. If each
mapping $f(\cdot,\omega):\X\longrightarrow\Y$ is $C$-concave in $\X$,
i.e.
$$
  f(tx_1+(1-t)x_2,\omega)-tf(x_1,\omega)-(1-t)f(x_2,\omega)\in C,
$$
for every $\omega\in\Omega$, then the set-valued mapping $\Phi_f:\X\rightrightarrows\Y$
defined by
$$
 \Phi_f(x)=\{y=f(x,\omega)\ |\ \omega\in\Omega\}=f(x,\Omega)
$$
turns out to be $C$-concave in $\X$. Indeed, taken any pair $x_1,\, x_2\in\X$
and $t\in [0,1]$, if $y$ is an arbitrary element of $\Phi_f(tx_1+(1-t)x_2)$,
then there exists $\omega\in\Omega$ such that
\begin{eqnarray*}
  y &=& f(tx_1+(1-t)x_2,\omega)\in tf(x_1,\omega)+(1-t)f(x_2,\omega)+C \\
  &\subseteq & t\Phi_f(x_1)+(1-t)\Phi_f(x_2)+C,
\end{eqnarray*}
which shows that
$$
  \Phi_f(tx_1+(1-t)x_2)\subseteq t\Phi_f(x_1)+(1-t)\Phi_f(x_2)+C.
$$
It is worth noting that if, in particular, $f:\X\times\Omega\longrightarrow\R^m$
is given by $f=(f_1,\dots,f_m)$, where $f_i(\cdot,\omega):\X\longrightarrow\R$
is concave for every $i=1,\dots,m$, and $\omega\in\Omega$,
then $f$ turns out to be $\R^m_+$-concave.
Thus, the set-valued inclusion $\Phi_f(x)\subseteq\R^m_+$ expresses in this case
the robust fulfilment of the convex inequality system
$$
  \left\{\begin{array}{c}
   -f_1(x,\omega)\le 0  \\
    \vdots \\
   -f_m(x,\omega)\le 0,
  \end{array}  \right.
$$
which typically defines the feasible region in robust convex optimization
(see \cite{BenNem98}).
\end{example}

\begin{remark}
The notion of $C$-concavity for set-valued mappings is evidently a generalization
of that of concavity, as presented in \cite[Definition 2.3]{Uder20}.
Consequently, several further examples of $C$-concave set-valued mappings,
including among other the class of fans introduced by Ioffe (see \cite{Ioff81}),
can be found therein.
\end{remark}

The $C$-concavity of $F$ in $\Pb\times\X$ yields the following important
geometric property of $\Solv$.

\begin{proposition}[Convexity of $\Solv$]    \label{pro:convSolv}
With reference to a parameterized set-valued inclusion $(\SVI_p)$,
suppose that $F:\Pb\times\X\rightrightarrows\Y$ is $C$-concave
on $\Pb\times\X$. Then, $\Solv:\Pb\rightrightarrows\X$ is a convex
set-valued mapping.
\end{proposition}

\begin{proof}
Taken arbitrary $p_1,\, p_2\in\dom\Solv$ and $x_1,\, x_2\in\X$,
with $x_i\in\Solv(p_i)$, $i=1,\, 2$, by virtue of the $C$-concavity
of $F$ one has
\begin{eqnarray*}
   F(t(p_1,x_1)+(1-t)(p_2,x_2)) &\subseteq & tF(p_1,x_1)+(1-t)F(p_2,x_2)+C \\
   &\subseteq & tC+(1-t)C+C \subseteq C,\quad\forall t\in [0,1].
\end{eqnarray*}
This inclusion shows that
$$
  tx_1+(1-t)x_2\in\Solv(tp_1+(1-t)p_2),\quad\forall t\in [0,1].
$$
By arbitrariness of $x_i\in\Solv(p_i)$, from the last inclusion one
can deduce that
$$
  t\Solv(p_1)+(1-t)\Solv(p_2)\subseteq\Solv(tp_1+(1-t)p_2),
  \quad\forall t\in [0,1].
$$
thereby completing the proof.
\end{proof}

As one expects, the convexity of $\Solv$, that is the convexity of
its graph, induces a similar geometric property in its graphical
approximation, as stated next.

\begin{corollary}[Sublinearity of $\Gder{\Solv}{\bar p,\bar x}$]
With reference to a parameterized set-valued inclusion $(\SVI_p)$,
let $(\bar p,\bar x)\in\graph\Solv$. If $F:\Pb\times\X\rightrightarrows\Y$
is $C$-concave on $\Pb\times\X$, then $\Gder{\Solv}{\bar p,\bar x}:\Pb
\rightrightarrows\X$ is a closed sublinear set-valued mapping
(a.k.a. convex process).

If, in addition, $\graph\Solv$ is closed, then $\Solv$ is protodifferentiable
at $(\bar p,\bar x)$.
\end{corollary}

\begin{proof}
Since by Proposition \ref{pro:convSolv} $\graph\Solv$ is convex, so is
$\Tang{\graph\Solv}{(\bar p,\bar x)}$. Having a closed, convex cone as
a graph, the set-valued mapping $\Gder{\Solv}{\bar p,\bar x}:\Pb
\rightrightarrows\X$ must be closed and sublinear.

As for the second assertion, it is useful to recall that a sufficient
condition for protodifferentiability is graph regularity
(see \cite[Proposition 8.41]{RocWet98}), which in turn
comes here as a consequence of the convexity of the graph along with
the outer semicontinuity (see \cite[Example 8.39]{RocWet98}),
the latter property being guaranteed by the fact that the graph
of $\Solv$ is closed (see \cite[Theorem 3B.2(c)]{DonRoc14}).
\end{proof}

Convexity interacts also with Aubin continuity yielding an enhanced
Lipschitzian behaviour of $\Solv$, according to the assertion below.

\begin{corollary}[Lipschitz continuity of $\Solv$ under truncation]
With reference to a parameterized set-valued inclusion $(\SVI_p)$,
let $(\bar p,\bar x)\in\graph\Solv$. Suppose that all the hypotheses
(i) -- (vi) of Proposition \ref{pro:vecterboAub} are fulfilled.
If $F:\Pb\times\X\rightrightarrows\Y$ is $C$-concave on $\Pb\times\X$,
then $\Solv$ has a Lipschitz continuous graphical localization
(not necessarily single-valued) around $(\bar p,\bar x)\in\graph\Solv$,
i.e. there exists neighbourhoods $V$ of $\bar p$ and $U$ of $\bar x$
such that the truncated mapping $p\leadsto \Solv(p)\cap U$ is Lipschitz
continuous on $V$.
\end{corollary}

\begin{proof}
Observe that, in the light of Proposition \ref{pro:convSolv}, $\Solv$
is convex. Consequently, it takes convex values. Since $\Solv$ is also
Aubin continuous at $(\bar p,\bar x)$ by virtue of the hypotheses taken,
the convexity of its images allows one to apply \cite[Theorem 3E.3]{DonRoc14},
after having noticed that this result can be extended to normed vector
spaces with the same proof.
\end{proof}

In the absence of convexity of $\Solv$, the following formulae provide
inner and outer approximations of $\Gder{\Solv}{\bar p,\bar x}$ in
terms of proper prederivatives.

\begin{theorem}[Inner approximation of $\Gder{\Solv}{\bar p,\bar x}$]   \label{thm:inapproxsol}
Given a parameterized set-valued inclusion $(\SVI_p)$, let $(\bar p,\bar x)
\in \Pb\times \X$. Suppose that:

(i) $\bar x\in\Solv(\bar p)$;

(ii) the set-valued mapping $p\leadsto F(p,\bar x)$ is Hausdorff $C$-u.s.c.
at $\bar p$;

(iii) there exists $\delta_1>0$ such that for every $p\in
\ball{\bar p}{\delta_1}$ each set-valued mapping $x\leadsto F(p,x)$
is l.s.c. on $\X$;

(iv) there exists $\delta_2>0$ such that, for every $p\in
\ball{\bar p}{\delta_2}$, each set-valued mapping $x\leadsto F(p,x)$
admits an outer prederivative $H_{F(p,\cdot)}(x;\cdot):\X\rightrightarrows\Y$
at each point $x\in \ball{\bar p}{\delta_2}$, such that $\sigma_H(\bar p,\bar x)>0$;

(v) $F$ admits an outer prederivative $H_F((\bar p,\bar x);\cdot):\Pb\times\X
\rightrightarrows\Y$ at $(\bar p,\bar x)$.

\noindent Then, the following approximation holds
\begin{equation}    \label{in:inapproxGderSol}
    \Gder{\Solv}{\bar p,\bar x}(p)\supseteq H_F^{+1}((\bar p,\bar x);(p,\cdot))(C),
    \quad\forall p\in\Pb.
\end{equation}
\end{theorem}

\begin{proof}
Take an arbitrary $v\in H_F^{+1}((\bar p,\bar x);(p,\cdot))(C)$.
Observe first of all that, since $\Gder{\Solv}{\bar p,\bar x}$ and
$H_F((\bar p,\bar x);\cdot)$ are both p.h. set-valued mappings,
it suffices to prove the validity of inclusion
$(\ref{in:inapproxGderSol})$ in the case $(p,v)\in\Uball\times\Uball$.
That said, in order to prove that $v\in \Gder{\Solv}{\bar p,\bar x}(p)$,
one has to show that $(p,v)\in\graph
\Gder{\Solv}{\bar p,\bar x}=\Tang{\graph\Solv}{(\bar p,\bar x)}$. On account of the
characterization recalled in $(\ref{eq:Tangchar})$, this can be done by showing that
$$
  \liminf_{t\downarrow 0}{\dist{(\bar p,\bar x)+t(p,v)}{\graph\Solv}\over t}=0.
$$
The last equality means that for every $\tau>0$ and $\epsilon>0$ there must
exist $t\in (0,\tau)$ such that
\begin{equation}   \label{in:distqed}
   {\dist{(\bar p,\bar x)+t(p,v)}{\graph\Solv}\over t}\le\epsilon.
\end{equation}
Fix positive $\epsilon$ and $\tau$. Since under the current hypotheses it is
possible to apply Proposition \ref{pro:vecterboAub}, one gets the existence of
$\zeta$ and $\eta$ such that the estimate
\begin{equation}    \label{in:erbobarpbarx}
     \dist{\bar x+tv}{\Solv(\bar p+tp)}\le{\nu_{F,C}(\bar p+tp,\bar x+tv)\over
     \sigma_H(\bar p,\bar x)},\quad \forall t\in (0,\min\{\zeta,\eta\})
\end{equation}
holds true.
By hypothesis (v), there exists $0<\delta_\epsilon<\min\{\zeta,\eta,\tau\}$ such that
$$
   F((\bar p,\bar x)+t(p,v)) \subseteq  F(\bar p,\bar x)+tH_F((\bar p,\bar x);
   (p,v))+\epsilon t\sigma_H(\bar p,\bar x)\Uball,\quad\forall t\in (0,\delta_\epsilon).
$$
As it is $v\in H_F((\bar p,\bar x);(p,v))\subseteq C$, from the above inclusion one
obtains
$$
   F((\bar p,\bar x)+t(p,v)) \subseteq F(\bar p,\bar x)+C+
   \epsilon t\sigma_H(\bar p,\bar x)\Uball,
   \quad\forall t\in (0,\delta_\epsilon).
$$
Thus, when passing to the excess, by virtue of what observed in Remark \ref{rem:provectexc}(i),
(ii) and (iii), one finds
\begin{eqnarray}    \label{in:nuepssigma}
   \nu_{F,C}((\bar p,\bar x)+t(p,v)) &\le & \exc{F(\bar p,\bar x)+C+  \nonumber
   \epsilon t\sigma_H(\bar p,\bar x)\Uball}{C}  \\
   &\le & \nu_{F,C}(\bar p,\bar x)+\epsilon t\sigma_H(\bar p,\bar x),\\
   &=& \epsilon t\sigma_H(\bar p,\bar x),\quad\forall t\in (0,\delta_\epsilon). \nonumber
\end{eqnarray}
Now, it is proper to observe that
\begin{eqnarray*}
  \dist{(\bar p,\bar x)+t(p,v)}{\graph\Solv} &=& \inf_{(q,w)\in\graph\Solv}
  \|(\bar p,\bar x)+t(p,v)-(q,w)\|   \\
   &\le& \inf_{w\in\Solv(\bar p+tp)}\|(\bar p,\bar x)+t(p,v)-(\bar p+tp,w)\|    \\
   &=& \inf_{w\in\Solv(\bar p+tp)}\|\bar x+tv-w\|=\dist{\bar x+tv}{\Solv(\bar p+tp)}.
\end{eqnarray*}
Therefore, by combining the last inequality chain with inequalities $(\ref{in:erbobarpbarx})$
and $(\ref{in:nuepssigma})$, one obtains
\begin{eqnarray*}
   \dist{(\bar p,\bar x)+t(p,v)}{\graph\Solv} \le {\nu_{F,C}(\bar p+tp,\bar x+tv)\over
     \sigma_H(\bar p,\bar x)}\le\epsilon t,
     \quad\forall t\in (0,\delta_\epsilon).
\end{eqnarray*}
As it is $\delta_\epsilon<\tau$, from the last inequality one can deduce the existence
of $t\in (0,\tau)$, for which inequality $(\ref{in:distqed})$ is fulfilled, thereby
completing the proof.
\end{proof}

\begin{remark}    \label{rem:inapproxGderSolv}
(i) The reader should notice that hypotheses (iv) and (v) of Theorem \ref{thm:inapproxsol}
speak about different mathematical objects. Indeed, the outer prederivatives mentioned
in hypothesis (iv) provide partial first-order approximations, each for the mapping
$F(p,\cdot)$ (depending only on $x$) near $\bar x$, with $p\in\ball{\bar p}{\delta_2}$.
In contrast, $H_F((\bar p,\bar x);\cdot)$ provides a first-order joint approximation of $F$
as a multifunction of both the variables $p$ and $x$.

(ii) It is readily seen that each set $H_F^{+1}((\bar p,\bar x);(p,\cdot))(C)$ is
a cone in $\X$, being the upper inverse image of a cone through a p.h. set-valued
mapping. In particular, whenever $H_F((\bar p,\bar x);\cdot)$ is $C$-concave, it can
be shown to be a convex cone. By virtue of these features, such a set is expected
to be more easily computable than the set $\Gder{\Solv}{\bar p,\bar x}(p)$, in the
spirit of implicit function theorems.
\end{remark}

\begin{theorem}[Outer approximation of $\Gder{\Solv}{\bar p,\bar x}$]   \label{thm:outapproxsol}
Given a parameterized set-valued inclusion $(\SVI_p)$, let $(\bar p,\bar x)
\in \Pb\times \X$, with $\bar x\in\Solv(\bar p)$.
Suppose that $F$ admits an inner prederivative
$H_F((\bar p,\bar x);\cdot):\Pb\times\X\rightrightarrows\Y$ at $(\bar p,\bar x)$,
which is l.s.c. on $\Pb\times\X$. Then, the following approximation holds
\begin{equation}     \label{in:outapproxsol}
   \Gder{\Solv}{\bar p,\bar x}(p)\subseteq \bigcap_{y\in F(\bar p,\bar x)}
   H_F^{+1}((\bar p,\bar x);(p,\cdot))(\Tang{C}{y}),\quad\forall p\in\Pb.
\end{equation}
\end{theorem}

\begin{proof}
Fix $p\in\Pb$ and take an arbitrary $v\in\Gder{\Solv}{\bar p,\bar x}(p)$.
According to the definition of
graphical (contingent) derivative, there exist sequences $(p_n)_n$ in $\Pb$,
with $p_n\longrightarrow p$, $(v_n)_n$ in $\X$, with $v_n\longrightarrow v$,
and $(t_n)_n$ in $(0,+\infty)$, with $t_n\downarrow 0$, such that
$\bar x+t_nv_n\in\Solv(\bar p+t_np_n)$ for every $n\in\N$, which means
\begin{equation}    \label{in:pntnvnC}
    F((\bar p,\bar x)+t_n(p_n,v_n))\subseteq C,\quad\forall n\in\N.
\end{equation}
Fix $\epsilon>0$. Since $H_F((\bar p,\bar x);\cdot)$ is an inner prederivative
of $F$ at $(\bar p,\bar x)$, there exists $\delta_\epsilon>0$ such that
\begin{equation}      \label{in:inprederoutaprox}
  F(\bar p,\bar x)+tH_F((\bar p,\bar x);(q,w))\subseteq F(p,x)+\epsilon t\Uball,
  \quad\forall (q,w)\in\Uball\times\Uball,\ \forall t\in (0,\delta_\epsilon).
\end{equation}
Since, without loss of generality, it is possible to assume that
$(p_n,v_n)\in\Uball\times\Uball$ and $t_n\in (0,\delta_\epsilon)$, by
combining inclusions $(\ref{in:inprederoutaprox})$ and $(\ref{in:pntnvnC})$,
one obtains
$$
  F(\bar p,\bar x)+t_nH_F((\bar p,\bar x);(p_n,v_n))\subseteq C+
  \epsilon t_n\Uball,\quad\forall n\in\N.
$$
Let $y$ be an arbitrary element of $F(\bar p,\bar x)$. As it is $F(\bar p,\bar x)
\subseteq C$ by hypothesis, the last inclusion implies
$$
  H_F((\bar p,\bar x);(p_n,v_n))\subseteq {C-y\over t_n}+\epsilon\Uball
  \subseteq \cl[\cone(C-y)]+\epsilon\Uball,\quad\forall n\in\N.
$$
Since $C$ is convex, by virtue of the representation $(\ref{eq:Tangrepconvex})$
valid for its contingent cone, the last inclusion gives
\begin{equation}      \label{in:pnvnNTCy}
  H_F((\bar p,\bar x);(p_n,v_n))\subseteq \Tang{C}{y}+\epsilon\Uball
  \subseteq \ball{\Tang{C}{y}}{\epsilon},\quad\forall n\in\N.
\end{equation}
As the set-valued mapping $H_F((\bar p,\bar x);\cdot)$ is l.s.c.
at $(p,v)$ and $\ball{\Tang{C}{y}}{\epsilon}$ is a closed set, from
the last inclusion it follows
\begin{equation}    \label{in:pvTCy}
    H_F((\bar p,\bar x);(p,v))\subseteq \ball{\Tang{C}{y}}{\epsilon}.
\end{equation}
Indeed, if one assumes that
$$
    H_F((\bar p,\bar x);(p,v))\cap\left[\Y\backslash\ball{\Tang{C}{y}}{\epsilon}\right]
    \ne\varnothing,
$$
there must exists $\delta_C$ such that
$$
    H_F((\bar p,\bar x);(q,w))\cap\left[\Y\backslash\ball{\Tang{C}{y}}{\epsilon}\right]
    \ne\varnothing,\quad\forall (q,w)\in\ball{p}{\delta_C}\times\ball{v}{\delta_C},
$$
which contradicts inclusion $(\ref{in:pnvnNTCy})$, because $(p_n,v_n)\longrightarrow (p,v)$
as $n\to\infty$, so that $(p_n,v_n)\in\ball{p}{\delta_C}\times \ball{v}{\delta_C}$.
By arbitrariness of $\epsilon$, as $\Tang{C}{y}$ is a closed set,
inclusion $(\ref{in:pvTCy})$ entails
$$
   H_F((\bar p,\bar x);(p,v))\subseteq\Tang{C}{y},
$$
what amounts to say
$$
  v\in H_F^{+1}((\bar p,\bar x);(p,\cdot))(\Tang{C}{y}).
$$
Since this reasoning is valid for every $y\in F(\bar p,\bar x)$,
one achieves the inclusion $(\ref{in:outapproxsol})$ in the thesis.
\end{proof}

For the set $\displaystyle\bigcap_{y\in F(\bar p,\bar x)}
H_F^{+1}((\bar p,\bar x);(p,\cdot))(\Tang{C}{y})$ what has been said
in Remark \ref{rem:inapproxGderSolv}(ii) can be repeated.

\begin{remark}
Since it is $\Tang{C}{y}=\Y$ whenever $y\in\inte C$ and $H_F^{+1}((\bar p,\bar x);
(p,\cdot))(\Y)=\X$, it should be clear that formula $(\ref{in:outapproxsol})$
yields an useful outer approximation of $\Gder{\Solv}{\bar p,\bar x}$,
provided that $F(\bar p,\bar x)\not\subseteq\inte C$.
On the other hand, in the case $F(\bar p,\bar x)\subseteq\inte C$,
if $F$ is u.s.c. at $(\bar p,\bar x)$, there exists $\delta_0>0$ such that
$$
  F(p,x)\subseteq C,\quad\forall (p,x)\in\ball{\bar p}{\delta_0}
  \times \ball{\bar x}{\delta_0}.
$$
This means that
$$
  \ball{\bar x}{\delta_0}\subseteq\Solv(p),\quad\forall
  p\in\ball{\bar p}{\delta_0}
$$
and hence $\ball{\bar p}{\delta_0}\times \ball{\bar x}{\delta_0}\subseteq
\graph\Solv$. In other terms, $(\bar p,\bar x)\in\inte\graph\Solv$.
Consequently, it results in $\Tang{\graph\Solv}{(\bar p,\bar x)}=
\Pb\times\X$. According to the definition of graphical (contingent)
derivative, this fact yields
$$
  \Gder{\Solv}{\bar p,\bar x}(p)=\X,\quad\forall p\in\Pb.
$$
So, the circumstance $F(\bar p,\bar x)\subseteq\inte C$, under an upper
semicontinuity assumption, turns out to be of less interest.
\end{remark}

Following a similar technique as in \cite[Theorem 5.5.1]{BorZhu05},
the following estimate of the coderivative of $\Solv$ can be
established.

\begin{proposition}[Coderivative of $\Solv$ via merit function]
Let $(\SVI_p)$ be a parameterized set-valued inclusion and let $(\bar p,\bar x)
\in\Pb\times\X$. Suppose that:

(i) $\bar x\in\Solv(\bar p)$;

(ii) the set-valued mapping $p\leadsto F(p,\bar x)$ is Hausdorff $C$-u.s.c.
at $\bar p$;

(iii) there exists $\delta_1>0$ such that for every $p\in
\ball{\bar p}{\delta_1}$ each set-valued mapping $x\leadsto F(p,x)$
is l.s.c. on $\X$;

(iv) there exists $\delta_2>0$ such that, for every $p\in
\ball{\bar p}{\delta_2}$, each set-valued mapping $x\leadsto F(p,x)$
admits an outer prederivative $H_{F(p,\cdot)}(x;\cdot):\X\rightrightarrows\Y$
at each point $x\in \ball{\bar p}{\delta_2}$, such that $\sigma_H(\bar p,\bar x)>0$.

\noindent Then, there exist positive $\zeta$ and $\eta$ such that the following
representation holds
\begin{eqnarray}     \label{eq:CoderSolvrep}
  \Coder{\Solv}{p,x}(x^*)=\{p^*\in\Pb^*\ & | &  (p^*,-x^*)\in\cone\Fsubd
  \nu_{F,C}(p,x))\}, \\
   & & \forall (p,x)\in [\inte\ball{\bar p}{\zeta}\times
  \inte\ball{\bar x}{\eta}]\cap\graph\Solv.   \nonumber
\end{eqnarray}
\end{proposition}

\begin{proof}
The above hypotheses ensure the validity of the first assertion in
Proposition \ref{pro:vecterboAub}. Thus, there exists positive $\zeta$ and $\eta$
such that
\begin{equation}     \label{in:erbovectSolvH}
  \dist{x}{\Solv(p)}\le{\nu_{F,C}(p,x)\over\sigma_H(\bar p,\bar x)},
    \quad\forall
    (p,x)\in\ball{\bar p}{\zeta}\times \ball{\bar x}{\eta}.
\end{equation}
Now, take an arbitrary pair $(p,x)\in [\inte\ball{\bar p}{\zeta}\times \inte
\ball{\bar x}{\eta}]\cap\graph\Solv$, $x^*\in\X^*$ and $p^*\in
\Coder{\Solv}{p,x}(x^*)$. According to the representation of a Fr\'echet
normal cone in terms of Fr\'echet subdifferential in $(\ref{eq:Fnconrep})$,
one has
$$
  (p^*,-x^*)\in\Ncone{\graph\Solv}{(p,x)}=\bigcup_{\kappa>0}\kappa\,
  \Fsubd\dist{\cdot}{\graph\Solv}(p,x).
$$
By applying a well-known variational description of Fr\'echet
subgradients (see, for instance, \cite[Theorem 1.88]{Mord06}),
the above inclusion means that, for some $\kappa>0$, there must exist
a function $\varphi\in C^1(\Pb\times\X)$ with $\Fder\varphi(p,x)=
(p^*,-x^*)$ and $\varphi(p,x)=\kappa\dist{(p,x)}{\graph\Solv}=0$,
such that
$$
  \varphi(q,z)\le\kappa\dist{(q,z)}{\graph\Solv},\quad\forall(q,z)
  \in\Pb\times\X.
$$
By combining the last inequality with inequality $(\ref{in:erbovectSolvH})$
and taking $r>0$ in such a way that $\ball{p}{r}\times\ball{x}{r}\subseteq
\ball{\bar p}{\zeta}\times\ball{\bar x}{\eta}$, one obtains
\begin{eqnarray*}
   \varphi(q,z) &\le & \varphi(p,x)+\kappa\dist{(q,z)}{\graph\Solv} \\
   &\le & \varphi(p,x)+\kappa\dist{z}{\Solv(q)}\le \varphi(p,x)+
   {\kappa\over\sigma_H(\bar p,\bar x)}\nu_{F,C}(q,z),   \\
   & & \quad\forall (q,z)\in \ball{p}{r}\times\ball{x}{r}.
\end{eqnarray*}
This inequality says that $(p,x)$ is a local minimizer of the function
$(q,z)\mapsto -\varphi(q,z)+ (\kappa/\sigma_H(\bar p,\bar x))\nu_{F,C}(q,z)$.
Consequently, by well-known calculus rules of the Fr\'echet subdifferential,
it must be
$$
  (p^*,-x^*)=\Fder\varphi(p,x)\in{\kappa\over\sigma_H(\bar p,\bar x)}
  \Fsubd\nu_{F,C}(p,x).
$$
Thus, the above argument proves the inclusion
$$
  \Coder{\Solv}{p,x}(x^*)\subseteq\{p^*\in\Pb^*\ |\  (p^*,-x^*)\in\cone\Fsubd
  \nu_{F,C}(p,x))\}.
$$
For proving the reverse inclusion it suffices to observe that for every
$\kappa>0$ it holds
$$
  \kappa\nu_{F,C}(q,z)\le\ind{(q,z)}{\graph\Solv},\quad\forall
  (q,z)\in\Pb\times\X,
$$
which, by known properties of the Fr\'echet subdifferential, implies
$$
  \kappa\Fsubd\nu_{F,C}(p,x)\subseteq\Fsubd\ind{\cdot}{\graph\Solv}
  (p,x)=\Ncone{\graph\Solv}{(p,x)},
$$
thereby completing the proof.
\end{proof}

\begin{remark}
Whenever $F$ is, in particular, $C$-concave in $\Pb\times\X$, function
$\nu_{F,C}$ turns out to be convex in $\Pb\times\X$. Indeed, for every
$(p_1,x_1),\, (p_2,x_2)\in\Pb\times\X$ and $t\in [0,1]$, by virtue of the
relations recalled in Remark \ref{rem:provectexc}(i) and (ii), it holds
\begin{eqnarray*}
    \nu_{F,C}(t(p_1,x_1)+(1-t)(p_2,x_2)) &=& \exc{F(t(p_1,x_1)+(1-t)(p_2,x_2))}{C} \\
    &\le& \exc{tF(t(p_1,x_1)+(1-t)F(p_2,x_2)+C}{C} \\
    &=& \exc{tF(p_1,x_1)+(1-t)F(p_2,x_2)}{C} \\
    &\le & t\exc{F(p_1,x_1)}{C}+(1-t)\exc{F(p_2,x_2)}{C}   \\
    &=& t\nu_{F,C}(p_1,x_1)+(1-t)\nu_{F,C}(p_2,x_2).
\end{eqnarray*}
Thus, in such an event the Fr\'echet subdifferential in formula
$(\ref{eq:CoderSolvrep})$ can be replaced with the subdifferential
in the sense of convex analysis.
\end{remark}


\section{Conclusions}   \label{Sect:5}

The findings exposed in Section \ref{Sect:3} and \ref{Sect:4} provide
some elements for a first-order analysis of multifunctions implicitly
defined by parameterized set-valued inclusions. These elements are
formulated in the language of modern variational analysis, speaking of
Lipschitzian properties and generalized derivatives. The methodology
employed for the main achievements is based on an error bound estimate,
which describes the metric behaviour of the solution mapping near points
of its graph, under proper infinitesimal conditions. Such an approach
leaves open some technical questions that could be matter for a future
deepening of the present research line. Among them, the following ones
are to be mentioned:

\begin{itemize}

\item[1.] Several results presented in Section \ref{Sect:4} invoke the
condition $\sigma_H(\bar p,\bar x)>0$. It would be useful to work out
this condition in relation to specific forms taken  by the outer
prederivative (e.g. fans and, in particular, those fans generated by
bundles of linear operators).

\item[2.] A general assumption on the set-valued inclusions considered
in Section \ref{Sect:4} is that $\inte C\ne\varnothing$. The author is
aware of the fact that this condition might be severe in the context of infinite-dimensional
spaces. It would be helpful therefore to devise surrogates of the condition $(\ref{eq:defsigmaH})$,
which allow one to avoid involving the topological interior of $C$.

\item[3.] In order to complete the analysis of $\Coder{\Solv}{p,x}$ with
a representation in terms of problem data, it would be proper to find how
to express $\Fsubd\nu_{F,C}(p,x)$ via the coderivative of $F$.

\end{itemize}

An impact evaluation of the main achievements on the treatment
of robust optimization problems deserves, of course, a dedicated analysis.

\vskip1cm



\begin{thebibliography}{99}

\bibitem{Aubi84} J.-P. Aubin, {\it Lipschitz behavior of solutions to convex
minimization problems}, Math. Oper. Res. {\bf 9}(1) (1984), 87--111.

\bibitem{AubFra09} J.-P. Aubin and H. Frankowska, {\it Set-valued analysis},
Birkh\"auser Boston, Boston, MA, 2009.

\bibitem{AzCoLu02} D. Az\'e, J.-N. Corvellec, and R.E. Lucchetti, {\it
Variational pairs and applications to stability in nonsmooth analysis},
Nonlinear Anal. {\bf 49}(5) (2002), 643--670.

\bibitem{BenNem98} A. Ben-Tal and A. Nemirovski, {\it Robust convex
optimization}, Math. Oper. Res. {\bf 23}(4) (1998),  769--805.

\bibitem{BorZhu05} J. M. Borwein and Q. J. Zhu, {\it Techniques of variational
analysis}, Springer-Verlag, New York, 2005.

\bibitem{Cast99} M. Castellani, {\it Error bounds for set-valued maps}, Generalized
convexity and optimization for economic and financial decisions,
121--135, Pitagora, Bologna, 1999.

\bibitem{DonRoc14} A.L. Dontchev and  R.T. Rockafellar, {\it Implicit
functions and solution mappings. A view from variational analysis},
Springer, New York, 2014.

\bibitem{DurStr12} M. Durea and R. Strugariu, {\it Openness stability and implicit
multifunction theorems: applications to variational systems}, Nonlinear Anal.
{\bf 75}(3) (2012), 1246--1259.

\bibitem{GfrOut16} H. Gfrerer and J.V. Outrata, {\it On Lipschitzian properties
of implicit multifunctions} SIAM J. Optim. {\bf 26}(4) (2016), 2160--2189. 

\bibitem{Ioff81} A.D. Ioffe, {\it Nonsmooth analysis: Differential calculus
of nondifferentiable mappings}, Trans. Amer. Math. Soc. {\bf 266}(1) (1981),
1--56.

\bibitem{Ioff00}  A.D. Ioffe, {\it Metric regularity and subdifferential calculus},
Uspekhi Mat. Nauk {\bf 55}(3) (2000), 103--162.

\bibitem{Ioff17} A.D. Ioffe, {\it Implicit functions: a metric theory},
Set-Valued Var. Anal. {\bf 25}(4) (2017), 679--699.

\bibitem{KhTaZa15} A.A. Khan, C. Tammer and C. Z\u alinescu, {\it Set-valued
optimization. An introduction with applications}, Springer, Heidelberg, 2015.

\bibitem{KlaKum02} D. Klatte and B. Kummer, {\it Nonsmooth equations in
optimizations. Regularity, calculus, methods and applications}, Kluwer
Academic Publishers, Dordrecht, 2002.

\bibitem{LedZhu99} Y.S. Ledyaev and Q.J. Zhu, {\it Implicit multifunction theorems},
Set-Valued Anal. {\bf 7}(3) (1999), 209--238.

\bibitem{LeTaYe08} G.M. Lee, N.N. Tam, and N.D. Yen, {\it Normal coderivative for
multifunctions and implicit function theorems}, J. Math. Anal. Appl. {\bf 338}(1)
(2008), 11--22. 

\bibitem{Mord94} B. S. Mordukhovich, {\it Stability theory for parametric generalized
equations and variational inequalities via nonsmooth analysis}, Trans. Amer. Math.
Soc. {\bf 343}(2), (1994), 609--657.

\bibitem{Mord94b} B. S. Mordukhovich, {\it Lipschitzian stability of constraint systems
and generalized equations}, Nonlinear Anal. {\bf 22}(2) (1994), 173--206.

\bibitem{Mord06} B. S. Mordukhovich, {\it Variational analysis and generalized
differentiation. I. Basic theory}, Springer-Verlag, Berlin, 2006.

\bibitem{NgaThe04} H. Van Ngai and M. Th\'era, {\it Error bounds and implicit
multifunction theorem in smooth Banach spaces and applications to optimization},
Set-Valued Anal. {\bf 12}(1--2) (2004), 195--223.

\bibitem{NgTrTh13} H. Van Ngai, N.H. Tron and M. Th\'era, {\it Implicit
multifunction theorems in complete metric spaces}, Math. Program. {\bf 139}(1--2)
(2013), Ser. B, 301--326.

\bibitem{Pang11} C.H.J. Pang, {\it Generalized Differentiation with Positively
Homogeneous Maps: Applications in Set-Valued Analysis and Metric Regularity},
Math. Oper. Res. {\bf 36}(3) (2011),  377--397.

\bibitem{Robi76} S.M. Robinson, {\it Stability theory for systems of inequalities.
II. Differentiable nonlinear systems}, SIAM J. Numer. Anal. {\bf 13}(4) (1976),
497--513.

\bibitem{Robi79} S.M. Robinson, {\it Generalized equations and their solutions.
I. Basic theory. Point-to-set maps and mathematical programming}, Math. Programming
Stud. No. {\bf 10} (1979), 128--141.

\bibitem{Robi91} S.M. Robinson, {\it An implicit-function theorem for a class
of nonsmooth functions}, Math. Oper. Res. {\bf 16}(2) (1991), 292--309.

\bibitem{RocWet98} R.T. Rockafellar and R.J.-B. Wets,
{\it Variational Analysis}, Springer-Verlag, Berlin, 1998.

\bibitem{RubVla00} A.M. Rubinov and A.A. Vladimirov, {\it Differences of convex
compacta and metric spaces of convex compacta with applications: a survey}, in
Quasidifferentiability and related topics, 263--296, Nonconvex Optim. Appl., 43,
Kluwer Acad. Publ., Dordrecht, 2000.

\bibitem{Schi07} W. Schirotzek, {\it Nonsmooth analysis}, Springer, Berlin, 2007.

\bibitem{Uder19} A. Uderzo, {\it On some generalized equations with metrically
$C$-increasing mappings: solvability and error bounds with applications to
optimization}, Optimization {\bf 68} (2019), 227--253.

\bibitem{Uder20} A. Uderzo, {\it Solution analysis for a class of set-inclusive
generalized equations: a convex analysis approach}, Pure Appl. Funct. Anal.
{\bf 5}(3) (2020),  769--790.

\bibitem{Uder20b} A. Uderzo, {\it On the quantitative solution stability
of parameterized set-valued inclusions}, Math. arXiv: 2003.0353v2 (2020),
1--21.

\end{thebibliography}
\end{document}